\documentclass[fleqn,twoside]{article}%
\usepackage{amsfonts}
\usepackage{amsmath}
\usepackage{amssymb}
\usepackage{graphicx}%
\providecommand{\U}[1]{\protect\rule{.1in}{.1in}}
\newtheorem{theorem}{Theorem}
\newtheorem{acknowledgement}[theorem]{Acknowledgement}

\newtheorem{corollary}[theorem]{Corollary}

\newtheorem{lemma}[theorem]{Lemma}

\newtheorem{remark}[theorem]{Remark}

\newenvironment{proof}[1][Proof]{\noindent\textbf{#1.} }{\ \rule{0.5em}{0.5em}}
\begin{document}
%
%TCIMACRO{\TeXButton{Title}{\title
%{Representation of solutions of discrete linear delay systems with non permutable matrices}%
%}}%
%BeginExpansion
\title
{Representation of solutions of discrete linear delay systems with non permutable matrices}%
%EndExpansion
%

%TCIMACRO{\TeXButton{Author}{\author{N. I . Mahmudov\\
%Department of Mathematics,
%Eastern Mediterranean University\\
%Famagusta, T.R. North Cyprus, via Mersin 10, Turkey}
%\date{}}}%
%BeginExpansion
\author{N. I . Mahmudov\\
Department of Mathematics,
Eastern Mediterranean University\\
Famagusta, T.R. North Cyprus, via Mersin 10, Turkey}
\date{}%
%EndExpansion
%

%TCIMACRO{\TeXButton{Make title}{\maketitle}}%
%BeginExpansion
\maketitle
%EndExpansion
%

%TCIMACRO{\TeXButton{Begin abstract}{\begin{abstract}} }%
%BeginExpansion
\begin{abstract}
%EndExpansion
We introduce a discrete delayed exponential depending on sequence of matrices.
This discrete matrix gives a representation of a solution to the Cauchy
problem for a discrete linear system with pure delay with sequence of
matrices. We discard the commutativity condition used in recent works related
to the representation of solutions for discrete delay linear systems.
%TCIMACRO{\TeXButton{End abstract}{\end{abstract}}}%
%BeginExpansion
\end{abstract}%
%EndExpansion


\section{Introduction}

For given $a,b\in\mathbb{Z\cup}\left\{  \pm\infty\right\}  $, $a<b$, we set
$Z_{a}^{b}:=\left\{  a,a+1,...,b\right\}  $. We study a discrete linear delay
system with sequence of matrices of the form:%

\begin{equation}
\left\{
\begin{array}
[c]{c}%
x\left(  k+1\right)  =Ax\left(  k\right)  +B_{k}x\left(  k-m\right)  +f\left(
k\right)  ,\\
x\left(  k\right)  =\varphi\left(  k\right)  ,\ \
\end{array}
\right.  \label{ds}%
\end{equation}
where $m\geq1$ is a fixed integer, $k\in Z_{0}^{\infty}$, $A=\left(
a_{ij}\right)  ,$ $\det A\neq0$ and $B_{k}=\left(  b_{ij}^{k}\right)  $ are a
constant $n\times n$ matrices, $f:Z_{0}^{\infty}\rightarrow R^{n}$ ,
$\varphi:Z_{-m}^{0}\rightarrow R^{n}$, $\Delta x(k)=x(k+1)-x(k)$. Solution
$x:Z_{-m}^{\infty}\rightarrow R^{n}$ of initial value problem is defined as an
infinite sequence $\left\{  \varphi\left(  -m\right)  ,\varphi\left(
-m+1\right)  ,...,\varphi\left(  0\right)  ,x\left(  1\right)  ,...,x\left(
k\right)  ,...\right\}  $ such that for any $k\in Z_{0}^{\infty},$ (\ref{ds}) holds.

Substituting in (\ref{ds}) $z\left(  k\right)  :=A^{-k}x\left(  k\right)  ,$
$D_{k}:=A^{-k-1}B_{k}A^{k-m},$ $k\in Z_{-m}^{\infty}$, we get an equivalent
discrete linear system of the form%
\begin{align}
z\left(  k+1\right)   &  =z\left(  k\right)  +D_{k}z\left(  k-m\right)
+A^{-k-1}f\left(  k\right)  ,\ \ \ k\in Z_{0}^{\infty},\label{ds1}\\
z\left(  k\right)   &  =A^{-k}\varphi\left(  k\right)  ,\ \ \ k\in Z_{-m}^{0}.
\label{ds2}%
\end{align}
Recently, Dibl\'{\i}k and Khusainov presented in \cite{diblik2},
\cite{diblik1} a solution of difference equations with linear parts with
constant coefficients given by permutable matrices and constant delay via a
discrete matrix delayed exponential. Advantage of discrete delayed exponential
matrix is to help transferring the classical idea to represent the solution of
linear ordinary differential equations into linear delay discrete equations.
Although there are many continued contributions in a discrete linear system
with pure delay with permutable matrices, to stability theory \cite{diblik5},
\cite{medved1}, \cite{medved2}, \cite{pospisil1}, \cite{li}, controllability
theory \cite{diblik7}, \cite{shuklin}, \cite{diblik8}, delay oscillating
systems \cite{diblik3}, discrete linear system with two delays \cite{diblik4},
\cite{diblik6}, Fredholm integral equations \cite{boichuk}, no results were
obtained for such systems with non permutable matrices. It should be mentioned
that recently non permutable case for the contiunous delay linear systems was
considered in \cite{medved3}.

We introduce a discrete delayed exponential depending on sequence of matrices
$\mathfrak{D}=\left\{  D_{1},D_{2},...\right\}  $ and give a representation of
solution to linear system of difference equations with delay parts with
nonconstant coefficients given by non permutable matrices. We discard the
commutativity condition used in recent works related to representation of
discrete delay linear system. In particular the results are new even for the
case when matrices $B_{k}$ does not depend on $k$, that is, $B_{k}=B$.

\section{Main results}

In order to drop the commutativity condition, we introduce the following
matrix%
\[
P^{\mathfrak{D}}\left(  k,d\right)  :=\left\{
\begin{tabular}
[c]{ll}%
$I,$ & $l=d=0,$\\
$%
%TCIMACRO{\dsum \limits_{j_{1}=\left(  d-1\right)  \left(  m+1\right)  }%
%^{k-1}}%
%BeginExpansion
{\displaystyle\sum\limits_{j_{1}=\left(  d-1\right)  \left(  m+1\right)
}^{k-1}}
%EndExpansion
D_{j_{1}}%
%TCIMACRO{\dsum \limits_{j_{2}=\left(  d-1\right)  \left(  m+1\right)  }%
%^{j_{1}}}%
%BeginExpansion
{\displaystyle\sum\limits_{j_{2}=\left(  d-1\right)  \left(  m+1\right)
}^{j_{1}}}
%EndExpansion
D_{j_{2}-m-1}...%
%TCIMACRO{\dsum \limits_{j_{d}=\left(  d-1\right)  \left(  m+1\right)
%}^{j_{d-1}}}%
%BeginExpansion
{\displaystyle\sum\limits_{j_{d}=\left(  d-1\right)  \left(  m+1\right)
}^{j_{d-1}}}
%EndExpansion
D_{j_{d}-\left(  d-1\right)  \left(  m+1\right)  },$ & $k\in Z_{\left(
d-1\right)  \left(  m+1\right)  +1}^{l\left(  m+1\right)  },$\\
& $l\in Z_{1}^{\infty},\ 1\leq d\leq l.$%
\end{tabular}
\right.  \
\]
We state and prove our first result. Note that in the proof of Lemma
\ref{lem:1} and Theorem \ref{thm:0} we follow the idea of the proofs of
statetmens in \cite{diblik2}.

\begin{lemma}
\label{lem:1}For any $l\in Z_{1}^{\infty},\ 1\leq d\leq l,$ $k\in Z_{\left(
d-1\right)  \left(  m+1\right)  +1}^{l\left(  m+1\right)  }$, the following
relation holds%
\begin{equation}
P^{\mathfrak{D}}\left(  k+1,d\right)  -P^{\mathfrak{D}}\left(  k,d\right)
=D_{k}P^{\mathfrak{D}}\left(  k-m,d-1\right)  . \label{rec1}%
\end{equation}

\end{lemma}

\begin{proof}
We will prove the lemma in two cases.

\textbf{Case 1: }$\left(  d-1\right)  \left(  m+1\right)  +1\leq k<l\left(
m+1\right)  $: In this case $k-m\in Z_{\left(  d-2\right)  \left(  m+1\right)
+1}^{\left(  l-1\right)  \left(  m+1\right)  }$ and by definition
\begin{align*}
P^{\mathfrak{D}}\left(  k-m,d-1\right)   &  =%
%TCIMACRO{\dsum \limits_{j_{1}=\left(  d-2\right)  \left(  m+1\right)
%}^{k-m-1}}%
%BeginExpansion
{\displaystyle\sum\limits_{j_{1}=\left(  d-2\right)  \left(  m+1\right)
}^{k-m-1}}
%EndExpansion
D_{j_{1}}%
%TCIMACRO{\dsum \limits_{j_{2}=\left(  d-2\right)  \left(  m+1\right)  }%
%^{j_{2}}}%
%BeginExpansion
{\displaystyle\sum\limits_{j_{2}=\left(  d-2\right)  \left(  m+1\right)
}^{j_{2}}}
%EndExpansion
D_{j_{2}-\left(  m+1\right)  }...%
%TCIMACRO{\dsum \limits_{j_{d-1}=\left(  d-2\right)  \left(  m+1\right)
%}^{j_{d-2}}}%
%BeginExpansion
{\displaystyle\sum\limits_{j_{d-1}=\left(  d-2\right)  \left(  m+1\right)
}^{j_{d-2}}}
%EndExpansion
D_{j_{d-1}-\left(  d-2\right)  \left(  m+1\right)  },\ 2\leq d\leq l,\\
P^{\mathfrak{D}}\left(  k-m,0\right)   &  =I.
\end{align*}
For $d=1$ we get $1\leq k<l\left(  m+1\right)  $ and
\[
P^{\mathfrak{D}}\left(  k+1,1\right)  -P^{\mathfrak{D}}\left(  k,1\right)  =%
%TCIMACRO{\dsum \limits_{j_{1}=0}^{k}}%
%BeginExpansion
{\displaystyle\sum\limits_{j_{1}=0}^{k}}
%EndExpansion
D_{j_{1}}-%
%TCIMACRO{\dsum \limits_{j_{1}=0}^{k-1}}%
%BeginExpansion
{\displaystyle\sum\limits_{j_{1}=0}^{k-1}}
%EndExpansion
D_{j_{1}}=D_{k}=D_{k}P^{\mathfrak{D}}\left(  k-m,0\right)  .
\]
For $2\leq d\leq l$ we get
\begin{align*}
&  P^{\mathfrak{D}}\left(  k+1,d\right)  -P^{\mathfrak{D}}\left(  k,d\right)
\\
&  =%
%TCIMACRO{\dsum \limits_{j_{1}=\left(  d-1\right)  \left(  m+1\right)  }^{k}}%
%BeginExpansion
{\displaystyle\sum\limits_{j_{1}=\left(  d-1\right)  \left(  m+1\right)  }%
^{k}}
%EndExpansion
D_{j_{1}}%
%TCIMACRO{\dsum \limits_{j_{2}=\left(  d-1\right)  \left(  m+1\right)  }%
%^{j_{1}}}%
%BeginExpansion
{\displaystyle\sum\limits_{j_{2}=\left(  d-1\right)  \left(  m+1\right)
}^{j_{1}}}
%EndExpansion
D_{j_{2}-m-1}...%
%TCIMACRO{\dsum \limits_{j_{d}=\left(  d-1\right)  \left(  m+1\right)
%}^{j_{d-1}}}%
%BeginExpansion
{\displaystyle\sum\limits_{j_{d}=\left(  d-1\right)  \left(  m+1\right)
}^{j_{d-1}}}
%EndExpansion
D_{j_{d}-\left(  d-1\right)  \left(  m+1\right)  }\\
&  -%
%TCIMACRO{\dsum \limits_{j_{1}=\left(  d-1\right)  \left(  m+1\right)  }%
%^{k-1}}%
%BeginExpansion
{\displaystyle\sum\limits_{j_{1}=\left(  d-1\right)  \left(  m+1\right)
}^{k-1}}
%EndExpansion
D_{j_{1}}%
%TCIMACRO{\dsum \limits_{j_{2}=\left(  d-1\right)  \left(  m+1\right)  }%
%^{j_{1}}}%
%BeginExpansion
{\displaystyle\sum\limits_{j_{2}=\left(  d-1\right)  \left(  m+1\right)
}^{j_{1}}}
%EndExpansion
D_{j_{2}-m-1}...%
%TCIMACRO{\dsum \limits_{j_{d}=\left(  d-1\right)  \left(  m+1\right)
%}^{j_{d-1}}}%
%BeginExpansion
{\displaystyle\sum\limits_{j_{d}=\left(  d-1\right)  \left(  m+1\right)
}^{j_{d-1}}}
%EndExpansion
D_{j_{d}-\left(  d-1\right)  \left(  m+1\right)  }\\
&  =D_{k}%
%TCIMACRO{\dsum \limits_{j_{2}=\left(  d-1\right)  \left(  m+1\right)  }^{k}}%
%BeginExpansion
{\displaystyle\sum\limits_{j_{2}=\left(  d-1\right)  \left(  m+1\right)  }%
^{k}}
%EndExpansion
D_{j_{2}-m-1}...%
%TCIMACRO{\dsum \limits_{j_{d}=\left(  d-1\right)  \left(  m+1\right)
%}^{j_{d-1}}}%
%BeginExpansion
{\displaystyle\sum\limits_{j_{d}=\left(  d-1\right)  \left(  m+1\right)
}^{j_{d-1}}}
%EndExpansion
D_{j_{d}-\left(  d-1\right)  \left(  m+1\right)  }\\
&  =D_{k}%
%TCIMACRO{\dsum \limits_{j_{1}=\left(  d-2\right)  \left(  m+1\right)
%}^{k-m-1}}%
%BeginExpansion
{\displaystyle\sum\limits_{j_{1}=\left(  d-2\right)  \left(  m+1\right)
}^{k-m-1}}
%EndExpansion
D_{j_{1}}%
%TCIMACRO{\dsum \limits_{j_{2}=\left(  d-2\right)  \left(  m+1\right)  }%
%^{j_{2}}}%
%BeginExpansion
{\displaystyle\sum\limits_{j_{2}=\left(  d-2\right)  \left(  m+1\right)
}^{j_{2}}}
%EndExpansion
D_{j_{2}-\left(  m+1\right)  }...%
%TCIMACRO{\dsum \limits_{j_{d-1}=\left(  d-2\right)  \left(  m+1\right)
%}^{j_{d-2}}}%
%BeginExpansion
{\displaystyle\sum\limits_{j_{d-1}=\left(  d-2\right)  \left(  m+1\right)
}^{j_{d-2}}}
%EndExpansion
D_{j_{d-1}-\left(  d-2\right)  \left(  m+1\right)  }\\
&  =D_{k}P^{\mathfrak{D}}\left(  k-m,d-1\right)  ,
\end{align*}
which proves the lemma for the case 1.

\textbf{Case 2:} $k=l(m+1)$: In this case $k+1=l(m+1)+1$ and $k-m=l\left(
m+1\right)  -m\in Z_{\left(  l-1\right)  \left(  m+1\right)  +1}^{l\left(
m+1\right)  }$. For $d=1$ we have:
\[
P^{\mathfrak{D}}\left(  l(m+1)+1,1\right)  -P^{\mathfrak{D}}\left(
l(m+1),1\right)  =%
%TCIMACRO{\dsum \limits_{j_{1}=0}^{l(m+1)}}%
%BeginExpansion
{\displaystyle\sum\limits_{j_{1}=0}^{l(m+1)}}
%EndExpansion
D_{j_{1}}-%
%TCIMACRO{\dsum \limits_{j_{1}=0}^{l(m+1)-1}}%
%BeginExpansion
{\displaystyle\sum\limits_{j_{1}=0}^{l(m+1)-1}}
%EndExpansion
D_{j_{1}}=D_{l(m+1)}=D_{l(m+1)}P^{\mathfrak{D}}\left(  l\left(  m+1\right)
-m,0\right)  .
\]
For $2\leq d\leq l$ we get%
\begin{align*}
&  P^{\mathfrak{D}}\left(  l(m+1)+1,d\right)  -P^{\mathfrak{D}}\left(
l(m+1),d\right) \\
&  =%
%TCIMACRO{\dsum \limits_{j_{1}=\left(  d-1\right)  \left(  m+1\right)
%}^{l(m+1)}}%
%BeginExpansion
{\displaystyle\sum\limits_{j_{1}=\left(  d-1\right)  \left(  m+1\right)
}^{l(m+1)}}
%EndExpansion
D_{j_{1}}%
%TCIMACRO{\dsum \limits_{j_{2}=\left(  d-1\right)  \left(  m+1\right)  }%
%^{j_{1}}}%
%BeginExpansion
{\displaystyle\sum\limits_{j_{2}=\left(  d-1\right)  \left(  m+1\right)
}^{j_{1}}}
%EndExpansion
D_{j_{2}-m-1}...%
%TCIMACRO{\dsum \limits_{j_{d}=\left(  d-1\right)  \left(  m+1\right)
%}^{j_{d-1}}}%
%BeginExpansion
{\displaystyle\sum\limits_{j_{d}=\left(  d-1\right)  \left(  m+1\right)
}^{j_{d-1}}}
%EndExpansion
D_{j_{d}-\left(  d-1\right)  \left(  m+1\right)  }\\
&  -%
%TCIMACRO{\dsum \limits_{j_{1}=\left(  d-1\right)  \left(  m+1\right)
%}^{l(m+1)-1}}%
%BeginExpansion
{\displaystyle\sum\limits_{j_{1}=\left(  d-1\right)  \left(  m+1\right)
}^{l(m+1)-1}}
%EndExpansion
D_{j_{1}}%
%TCIMACRO{\dsum \limits_{j_{2}=\left(  d-1\right)  \left(  m+1\right)  }%
%^{j_{1}}}%
%BeginExpansion
{\displaystyle\sum\limits_{j_{2}=\left(  d-1\right)  \left(  m+1\right)
}^{j_{1}}}
%EndExpansion
D_{j_{2}-m-1}...%
%TCIMACRO{\dsum \limits_{j_{d}=\left(  d-1\right)  \left(  m+1\right)
%}^{j_{d-1}}}%
%BeginExpansion
{\displaystyle\sum\limits_{j_{d}=\left(  d-1\right)  \left(  m+1\right)
}^{j_{d-1}}}
%EndExpansion
D_{j_{d}-\left(  d-1\right)  \left(  m+1\right)  }\\
&  =D_{l\left(  m+1\right)  }%
%TCIMACRO{\dsum \limits_{j_{2}=\left(  d-1\right)  \left(  m+1\right)
%}^{l\left(  m+1\right)  }}%
%BeginExpansion
{\displaystyle\sum\limits_{j_{2}=\left(  d-1\right)  \left(  m+1\right)
}^{l\left(  m+1\right)  }}
%EndExpansion
D_{j_{2}-m-1}%
%TCIMACRO{\dsum \limits_{j_{3}=\left(  d-1\right)  \left(  m+1\right)  }%
%^{j_{2}}}%
%BeginExpansion
{\displaystyle\sum\limits_{j_{3}=\left(  d-1\right)  \left(  m+1\right)
}^{j_{2}}}
%EndExpansion
D_{j_{3}-2\left(  m+1\right)  }...%
%TCIMACRO{\dsum \limits_{j_{d}=\left(  d-1\right)  \left(  m+1\right)
%}^{j_{d-1}}}%
%BeginExpansion
{\displaystyle\sum\limits_{j_{d}=\left(  d-1\right)  \left(  m+1\right)
}^{j_{d-1}}}
%EndExpansion
D_{j_{d}-\left(  d-1\right)  \left(  m+1\right)  }\\
&  =D_{l\left(  m+1\right)  }%
%TCIMACRO{\dsum \limits_{j_{1}=\left(  d-2\right)  \left(  m+1\right)
%}^{\left(  l-1\right)  \left(  m+1\right)  }}%
%BeginExpansion
{\displaystyle\sum\limits_{j_{1}=\left(  d-2\right)  \left(  m+1\right)
}^{\left(  l-1\right)  \left(  m+1\right)  }}
%EndExpansion
D_{j_{1}}%
%TCIMACRO{\dsum \limits_{j_{2}=\left(  d-2\right)  \left(  m+1\right)  }%
%^{j_{1}}}%
%BeginExpansion
{\displaystyle\sum\limits_{j_{2}=\left(  d-2\right)  \left(  m+1\right)
}^{j_{1}}}
%EndExpansion
D_{j_{2}-\left(  m+1\right)  }...%
%TCIMACRO{\dsum \limits_{j_{d-1}=\left(  d-2\right)  \left(  m+1\right)
%}^{j_{d-2}}}%
%BeginExpansion
{\displaystyle\sum\limits_{j_{d-1}=\left(  d-2\right)  \left(  m+1\right)
}^{j_{d-2}}}
%EndExpansion
D_{j_{d-1}-\left(  d-2\right)  \left(  m+1\right)  }\\
&  =D_{l\left(  m+1\right)  }P^{\mathfrak{D}}\left(  l\left(  m+1\right)
-m,d-1\right)  ,\ 2\leq d\leq l.
\end{align*}

\end{proof}

Using $P\left(  k,d\right)  $ we may define the delayed exponential depending
on sequence of matrices$:$%
\begin{equation}
e_{m}^{\mathfrak{D}}\left(  k\right)  =\left\{
\begin{tabular}
[c]{ll}%
$\Theta$ & $k\in Z_{-\infty}^{-m-1},$\\
$I$ & $k\in Z_{-m}^{0},$\\
$I+%
%TCIMACRO{\dsum \limits_{d=1}^{l}}%
%BeginExpansion
{\displaystyle\sum\limits_{d=1}^{l}}
%EndExpansion
P\left(  k,d\right)  $ & $k\in Z_{\left(  l-1\right)  \left(  m+1\right)
+1}^{l\left(  m+1\right)  },\ \ l\in Z_{1}^{\infty}.$%
\end{tabular}
\ \ \ \ \ \ \ \ \ \ \ \ \ \right.  \label{exp}%
\end{equation}

\begin{theorem}
\label{thm:0}For any $k\in Z_{\left(  l-1\right)  \left(  m+1\right)
+1}^{l\left(  m+1\right)  }$, the following relation holds
\begin{equation}
e_{m}^{\mathfrak{D}}\left(  k+1\right)  -e_{m}^{\mathfrak{D}}\left(  k\right)
=D_{k}e_{m}^{\mathfrak{D}}\left(  k-m\right)  ,\ \ k\in Z_{\left(  l-1\right)
\left(  m+1\right)  +1}^{l\left(  m+1\right)  -1}. \label{rr11}%
\end{equation}

\end{theorem}

\begin{proof}
The proof is based on Lemma \ref{lem:1}. Let us consider the cases when
$\left(  l-1\right)  \left(  m+1\right)  +1\leq k<l\left(  m+1\right)  $ and
$k=l(m+1).$

Case 1: $\left(  l-1\right)  \left(  m+1\right)  +1\leq k<l\left(  m+1\right)
$: By Lemma \ref{lem:1},
\begin{align*}
e_{m}^{\mathfrak{D}}\left(  k+1\right)  -e_{m}^{\mathfrak{D}}\left(  k\right)
&  =%
%TCIMACRO{\dsum \limits_{d=1}^{l}}%
%BeginExpansion
{\displaystyle\sum\limits_{d=1}^{l}}
%EndExpansion
P^{\mathfrak{D}}\left(  k+1,d\right)  -%
%TCIMACRO{\dsum \limits_{d=1}^{l}}%
%BeginExpansion
{\displaystyle\sum\limits_{d=1}^{l}}
%EndExpansion
P^{\mathfrak{D}}\left(  k,d\right)  =D_{k}%
%TCIMACRO{\dsum \limits_{d=1}^{l}}%
%BeginExpansion
{\displaystyle\sum\limits_{d=1}^{l}}
%EndExpansion
P^{\mathfrak{D}}\left(  k-m,d-1\right) \\
&  =D_{k}%
%TCIMACRO{\dsum \limits_{d=0}^{l-1}}%
%BeginExpansion
{\displaystyle\sum\limits_{d=0}^{l-1}}
%EndExpansion
P^{\mathfrak{D}}\left(  k-m,d\right)  =D_{k}e_{m}^{\mathfrak{D}}\left(
k-m\right)  .\ \ \ k-m\in Z_{\left(  l-2\right)  \left(  m+1\right)
+1}^{\left(  l-1\right)  \left(  m+1\right)  }.
\end{align*}
Case 2: $k=l(m+1)$: By Lemma \ref{lem:1},%
\begin{align*}
e_{m}^{\mathfrak{D}}\left(  k+1\right)  -e_{m}^{\mathfrak{D}}\left(  k\right)
&  =%
%TCIMACRO{\dsum \limits_{d=1}^{l+1}}%
%BeginExpansion
{\displaystyle\sum\limits_{d=1}^{l+1}}
%EndExpansion
P^{\mathfrak{D}}\left(  l(m+1)+1,d\right)  -%
%TCIMACRO{\dsum \limits_{d=1}^{l}}%
%BeginExpansion
{\displaystyle\sum\limits_{d=1}^{l}}
%EndExpansion
P^{\mathfrak{D}}\left(  l(m+1),d\right) \\
&  =%
%TCIMACRO{\dsum \limits_{d=1}^{l}}%
%BeginExpansion
{\displaystyle\sum\limits_{d=1}^{l}}
%EndExpansion
\left(  P^{\mathfrak{D}}\left(  l(m+1)+1,d\right)  -P^{\mathfrak{D}}\left(
l(m+1),d\right)  \right)  +P^{\mathfrak{D}}\left(  l(m+1)+1,l+1\right) \\
&  =D_{l\left(  m+1\right)  }%
%TCIMACRO{\dsum \limits_{d=1}^{l}}%
%BeginExpansion
{\displaystyle\sum\limits_{d=1}^{l}}
%EndExpansion
P^{\mathfrak{D}}\left(  l\left(  m+1\right)  -m,d-1\right)  +D_{l\left(
m+1\right)  }D_{\left(  l-1\right)  \left(  m+1\right)  }...D_{0}\\
&  =D_{l\left(  m+1\right)  }\left(  I+%
%TCIMACRO{\dsum \limits_{d=1}^{l-1}}%
%BeginExpansion
{\displaystyle\sum\limits_{d=1}^{l-1}}
%EndExpansion
P^{\mathfrak{D}}\left(  l\left(  m+1\right)  -m,d\right)  +P^{\mathfrak{D}%
}\left(  l\left(  m+1\right)  -m,l\right)  \right) \\
&  =D_{l\left(  m+1\right)  }e_{m}^{\mathfrak{D}}\left(  k-m\right)  .
\end{align*}
Here we used the following formula
\begin{align*}
P^{\mathfrak{D}}\left(  l(m+1)+1,l+1\right)   &  =%
%TCIMACRO{\dsum \limits_{j_{1}=l\left(  m+1\right)  }^{l(m+1)}}%
%BeginExpansion
{\displaystyle\sum\limits_{j_{1}=l\left(  m+1\right)  }^{l(m+1)}}
%EndExpansion
D_{j_{1}}%
%TCIMACRO{\dsum \limits_{j_{2}=l\left(  m+1\right)  }^{j_{1}}}%
%BeginExpansion
{\displaystyle\sum\limits_{j_{2}=l\left(  m+1\right)  }^{j_{1}}}
%EndExpansion
D_{j_{2}-m-1}...%
%TCIMACRO{\dsum \limits_{j_{l+1}=l\left(  m+1\right)  }^{j_{l}}}%
%BeginExpansion
{\displaystyle\sum\limits_{j_{l+1}=l\left(  m+1\right)  }^{j_{l}}}
%EndExpansion
D_{j_{l+1}-l\left(  m+1\right)  }\\
&  =D_{l\left(  m+1\right)  }%
%TCIMACRO{\dsum \limits_{j_{2}=l\left(  m+1\right)  }^{l\left(  m+1\right)  }}%
%BeginExpansion
{\displaystyle\sum\limits_{j_{2}=l\left(  m+1\right)  }^{l\left(  m+1\right)
}}
%EndExpansion
D_{j_{2}-m-1}...%
%TCIMACRO{\dsum \limits_{j_{l+1}=l\left(  m+1\right)  }^{j_{l}}}%
%BeginExpansion
{\displaystyle\sum\limits_{j_{l+1}=l\left(  m+1\right)  }^{j_{l}}}
%EndExpansion
D_{j_{l+1}-l\left(  m+1\right)  }\\
&  =D_{l\left(  m+1\right)  }D_{\left(  l-1\right)  \left(  m+1\right)
}...D_{0}.
\end{align*}

\end{proof}

Now using the matrix delayed exponential we can solve the following linear
matrix equation:%
\begin{align}
\Phi\left(  k+1\right)   &  =A\Phi\left(  k\right)  +D_{k}\Phi\left(
k-m\right)  ,\ k\in Z_{0}^{\infty},\nonumber\\
X\left(  k\right)   &  =A^{k},\ \ \ k\in Z_{-m}^{0}. \label{ds3}%
\end{align}

\begin{theorem}
\label{thm:1} The matrix%
\[
\Phi\left(  k\right)  =\left\{
\begin{tabular}
[c]{lll}%
$\Theta$ & $\text{if\ }$ & $k\in Z_{-\infty}^{-m-1},$\\
$A^{k}$ & $\text{if\ }$ & $k\in Z_{-m}^{0},$\\
$A^{k}\left(  I+%
%TCIMACRO{\dsum \limits_{j=1}^{l}}%
%BeginExpansion
{\displaystyle\sum\limits_{j=1}^{l}}
%EndExpansion
P\left(  k,j\right)  \right)  $ & $\text{if\ }$ & $k\in Z_{\left(  l-1\right)
\left(  m+1\right)  +1}^{l\left(  m+1\right)  }.$%
\end{tabular}
\ \ \ \ \ \right.
\]
solves the problem (\ref{ds3}).
\end{theorem}

\begin{remark}
It can be shown that%
\[%
%TCIMACRO{\dsum \limits_{j_{1}=\left(  d-1\right)  \left(  m+1\right)  }%
%^{k-1}}%
%BeginExpansion
{\displaystyle\sum\limits_{j_{1}=\left(  d-1\right)  \left(  m+1\right)
}^{k-1}}
%EndExpansion%
%TCIMACRO{\dsum \limits_{j_{2}=\left(  d-1\right)  \left(  m+1\right)  }%
%^{j_{1}}}%
%BeginExpansion
{\displaystyle\sum\limits_{j_{2}=\left(  d-1\right)  \left(  m+1\right)
}^{j_{1}}}
%EndExpansion
...%
%TCIMACRO{\dsum \limits_{j_{d}=\left(  d-1\right)  \left(  m+1\right)
%}^{j_{d-1}}}%
%BeginExpansion
{\displaystyle\sum\limits_{j_{d}=\left(  d-1\right)  \left(  m+1\right)
}^{j_{d-1}}}
%EndExpansion
1=\left(
\begin{array}
[c]{c}%
k-\left(  d-1\right)  m\\
d
\end{array}
\right)  ,\ \ \ \ \ \ \ k\in Z_{\left(  d-1\right)  \left(  m+1\right)
+1}^{l\left(  m+1\right)  }.
\]
If $B_{k}$ does not depend on $k$, $B_{k}=B$ and matrices $A$ and $B$ are
permutable ($AB=BA$), then $D_{k}=A^{-k-1}BA^{k-m}=A^{-1}BA^{-m}=:D$ and
\[
P^{\mathfrak{D}}\left(  k,d\right)  =A^{-d}B^{d}A^{-dm}%
%TCIMACRO{\dsum \limits_{j_{1}=\left(  d-1\right)  \left(  m+1\right)  }%
%^{k-1}}%
%BeginExpansion
{\displaystyle\sum\limits_{j_{1}=\left(  d-1\right)  \left(  m+1\right)
}^{k-1}}
%EndExpansion%
%TCIMACRO{\dsum \limits_{j_{2}=\left(  d-1\right)  \left(  m+1\right)  }%
%^{j_{1}}}%
%BeginExpansion
{\displaystyle\sum\limits_{j_{2}=\left(  d-1\right)  \left(  m+1\right)
}^{j_{1}}}
%EndExpansion
...%
%TCIMACRO{\dsum \limits_{j_{d}=\left(  d-1\right)  \left(  m+1\right)
%}^{j_{d-1}}}%
%BeginExpansion
{\displaystyle\sum\limits_{j_{d}=\left(  d-1\right)  \left(  m+1\right)
}^{j_{d-1}}}
%EndExpansion
1=A^{-d}B^{d}A^{-dm}\left(
\begin{array}
[c]{c}%
k-\left(  d-1\right)  m\\
d
\end{array}
\right)  .
\]
In this case
\[
e_{m}^{\mathfrak{D}}\left(  k\right)  =\left\{
\begin{tabular}
[c]{ll}%
$\Theta$ & $k\in Z_{-\infty}^{-m-1},$\\
$I$ & $k\in Z_{-m}^{0},$\\
$I+%
%TCIMACRO{\dsum \limits_{d=1}^{l}}%
%BeginExpansion
{\displaystyle\sum\limits_{d=1}^{l}}
%EndExpansion
A^{-d}B^{d}A^{-dm}\left(
\begin{array}
[c]{c}%
k-\left(  d-1\right)  m\\
d
\end{array}
\right)  ,$ & $k\in Z_{\left(  l-1\right)  \left(  m+1\right)  +1}^{l\left(
m+1\right)  },$%
\end{tabular}
\ \ \ \ \ \ \ \ \ \ \ \ \ \ \right.
\]
and coincides with the discrete matrix delayed exponential defined in
\cite{diblik2}.
\end{remark}

Using $\Phi\left(  k\right)  $ we give the representation of solution to the
homogeneous delay problem%
\begin{align}
x\left(  k+1\right)   &  =Ax\left(  k\right)  +B_{k}x\left(  k-m\right)
,\ \ \ k\in Z_{0}^{\infty},\label{q1}\\
x\left(  k\right)   &  =\varphi\left(  k\right)  ,\ \ k\in Z_{-m}^{0}.
\label{q2}%
\end{align}

\begin{theorem}
The solution of the problem (\ref{q1}), (\ref{q2}) can be expressed as%
\begin{equation}
x\left(  k\right)  =\Phi\left(  k\right)  A^{-m}\varphi\left(  -m\right)
+A^{m}%
%TCIMACRO{\dsum \limits_{j=-m+1}^{0}}%
%BeginExpansion
{\displaystyle\sum\limits_{j=-m+1}^{0}}
%EndExpansion
\Phi\left(  k-m-j\right)  \left(  \varphi\left(  j\right)  -A\varphi\left(
j-1\right)  \right)  . \label{rep1}%
\end{equation}

\end{theorem}

\begin{proof}
Introduce a new variable $z\left(  k\right)  =A^{-k}x\left(  k\right)  ,$
$k\in Z_{-m}^{\infty}$. Then the problem (\ref{q1}), (\ref{q2}) is equivalent
to%
\begin{align}
z\left(  k+1\right)   &  =z\left(  k\right)  +D_{k}z\left(  k-m\right)
,\ \ \ k\in Z_{0}^{\infty},\label{qw2}\\
z\left(  k\right)   &  =A^{-k}\varphi\left(  k\right)  ,\ \ \ k\in Z_{-m}^{0}.
\label{qw1}%
\end{align}
We are looking for the representation of solution of the problem (\ref{qw2}),
(\ref{qw1}) in the form%
\begin{equation}
z\left(  k\right)  =e_{m}^{\mathfrak{D}}\left(  k\right)  C+%
%TCIMACRO{\dsum \limits_{j=-m+1}^{0}}%
%BeginExpansion
{\displaystyle\sum\limits_{j=-m+1}^{0}}
%EndExpansion
e_{m}^{\mathfrak{D}}\left(  k-m-j\right)  \omega\left(  j\right)  ,\ \ k\in
Z_{-m}^{\infty}, \label{q3}%
\end{equation}
where $C\in R^{n}$ is an unknown vector and $\omega:Z_{-m}^{0}\rightarrow
R^{n}$ is an unknown discrete function. The representation (\ref{q3}) is a
solution of homogeneous delay equation (\ref{qw2}) for any $C$ and $\omega$
and for $k\in Z_{0}^{\infty}$. Indeed, by the formula (\ref{rr11}) we get
\begin{align*}
\Delta z\left(  k\right)   &  =\Delta\left[  e_{m}^{\mathfrak{D}}\left(
k\right)  \right]  C+%
%TCIMACRO{\dsum \limits_{j=-m+1}^{0}}%
%BeginExpansion
{\displaystyle\sum\limits_{j=-m+1}^{0}}
%EndExpansion
\Delta\left[  e_{m}^{\mathfrak{D}}\left(  k-m-j\right)  \right]  \omega\left(
j\right) \\
&  =D_{k}\left[  e_{m}^{\mathfrak{D}}\left(  k-m\right)  C+%
%TCIMACRO{\dsum \limits_{k=-m+1}^{0}}%
%BeginExpansion
{\displaystyle\sum\limits_{k=-m+1}^{0}}
%EndExpansion
e_{m}^{\mathfrak{D}}\left(  k-2m-j\right)  \omega\left(  j\right)  \right] \\
&  =D_{k}z\left(  k-m\right)  ,\ \ \ k\in Z_{0}^{\infty}.
\end{align*}
So expression (\ref{q3}) solves (\ref{qw2}) for $k\in Z_{0}^{\infty}$. Now we
determine $C$ and $\omega$. By definition, $C$ and $\omega$ must satisfy the
initial condition (\ref{qw1}) for $k\in Z_{-m}^{0}$. For $k\in Z_{-m}^{0},$
(\ref{q3}) leads to relation%
\begin{align*}
z\left(  k\right)   &  =A^{-k}\varphi\left(  k\right)  =e_{m}^{\mathfrak{D}%
}\left(  k\right)  C+%
%TCIMACRO{\dsum \limits_{j=-m+1}^{0}}%
%BeginExpansion
{\displaystyle\sum\limits_{j=-m+1}^{0}}
%EndExpansion
e_{m}^{\mathfrak{D}}\left(  k-m-j\right)  \omega\left(  j\right) \\
&  =e_{m}^{\mathfrak{D}}\left(  k\right)  C+%
%TCIMACRO{\dsum \limits_{j=-m+1}^{k}}%
%BeginExpansion
{\displaystyle\sum\limits_{j=-m+1}^{k}}
%EndExpansion
e_{m}^{\mathfrak{D}}\left(  k-m-j\right)  \omega\left(  j\right)  +%
%TCIMACRO{\dsum \limits_{j=k+1}^{0}}%
%BeginExpansion
{\displaystyle\sum\limits_{j=k+1}^{0}}
%EndExpansion
e_{m}^{\mathfrak{D}}\left(  k-m-j\right)  \omega\left(  j\right) \\
&  =C+%
%TCIMACRO{\dsum \limits_{j=-m+1}^{k}}%
%BeginExpansion
{\displaystyle\sum\limits_{j=-m+1}^{k}}
%EndExpansion
e_{m}^{\mathfrak{D}}\left(  k-m-j\right)  \omega\left(  j\right)  +%
%TCIMACRO{\dsum \limits_{j=k+1}^{0}}%
%BeginExpansion
{\displaystyle\sum\limits_{j=k+1}^{0}}
%EndExpansion
e_{m}^{\mathfrak{D}}\left(  k-m-j\right)  \omega\left(  j\right)  ,\ \ k\in
Z_{-m}^{0}.
\end{align*}
It follows that
\begin{align*}
A^{m}\varphi\left(  -m\right)   &  =C,\ \ k=-m,\\
A^{-k}\varphi\left(  k\right)   &  =C+%
%TCIMACRO{\dsum \limits_{j=-m+1,j\leq k}^{k}}%
%BeginExpansion
{\displaystyle\sum\limits_{j=-m+1,j\leq k}^{k}}
%EndExpansion
\omega\left(  j\right)  ,\ \ k\in Z_{-m+1}^{0},
\end{align*}
and one can obtain%
\begin{align*}
\omega\left(  k\right)   &  =A^{-k}\varphi\left(  k\right)  -A^{-k+1}%
\varphi\left(  k-1\right)  ,\ \ k=-m,-m+1,...,0,\\
C  &  =A^{-m}\varphi\left(  -m\right)  .
\end{align*}
In order to get the formula (\ref{rep1}), it remains to insert $C$ and
$\omega$ into (\ref{q3}). Indeed%
\begin{align*}
x\left(  k\right)   &  =A^{k}z\left(  k\right)  =A^{k}\left(  e_{m}%
^{\mathfrak{D}}\left(  k\right)  A^{-m}\varphi\left(  -m\right)  +%
%TCIMACRO{\dsum \limits_{j=-m+1}^{0}}%
%BeginExpansion
{\displaystyle\sum\limits_{j=-m+1}^{0}}
%EndExpansion
e_{m}^{\mathfrak{D}}\left(  k-m-j\right)  A^{-j}\left(  \varphi\left(
j\right)  -A\varphi\left(  j-1\right)  \right)  \right) \\
&  =\Phi\left(  k\right)  A^{-m}\varphi\left(  -m\right)  +A^{k}%
%TCIMACRO{\dsum \limits_{j=-m+1}^{0}}%
%BeginExpansion
{\displaystyle\sum\limits_{j=-m+1}^{0}}
%EndExpansion
A^{-k+m+j}A^{k-m-j}\Phi\left(  k-m-j\right)  A^{-j}\left(  \varphi\left(
j\right)  -A\varphi\left(  j-1\right)  \right) \\
&  =\Phi\left(  k\right)  A^{-m}\varphi\left(  -m\right)  +%
%TCIMACRO{\dsum \limits_{j=-m+1}^{0}}%
%BeginExpansion
{\displaystyle\sum\limits_{j=-m+1}^{0}}
%EndExpansion
A^{m+j}\Phi\left(  k-m-j\right)  A^{-j}\left(  \varphi\left(  j\right)
-A\varphi\left(  j-1\right)  \right)  .
\end{align*}

\end{proof}

\begin{corollary}
A solution $x:Z_{-m}^{\infty}\rightarrow R^{n}$ of initial value problem
(\ref{ds}) has a form%
\begin{align}
x\left(  k\right)   &  =\Phi\left(  k\right)  A^{-m}\varphi\left(  -m\right)
+%
%TCIMACRO{\dsum \limits_{j=-m+1}^{0}}%
%BeginExpansion
{\displaystyle\sum\limits_{j=-m+1}^{0}}
%EndExpansion
A^{m+j}\Phi\left(  k-m-j\right)  A^{-j}\left(  \varphi\left(  j\right)
-A\varphi\left(  j-1\right)  \right) \nonumber\\
&  +%
%TCIMACRO{\dsum \limits_{j=1}^{k}}%
%BeginExpansion
{\displaystyle\sum\limits_{j=1}^{k}}
%EndExpansion
A^{m+j}\Phi\left(  k-m-j\right)  A^{-j}f\left(  j-1\right)  . \label{ww1}%
\end{align}

\end{corollary}

\begin{remark}
If $B_{k}$ does not depend on $k$, that is, $B_{k}=B$ and matrices $A$ and $B$
are permutable ($AB=BA$), then $D_{k}=A^{-k-1}BA^{k-m}=A^{-1}BA^{-m}=:D$ then
the presentation (\ref{ww1}) coincides with the formula obtained in
\cite{diblik2}.
\end{remark}

\begin{acknowledgement}
I would like to thank the reviewers for giving constructive comments and
suggestions which would help me to improve the quality of the paper.
\end{acknowledgement}

\end{document}